\documentclass[twoside,leqno]{article}

\usepackage[letterpaper]{geometry}

\usepackage{siamproceedings}

\usepackage{amsfonts,amssymb, amsopn,colonequals}
\usepackage{epstopdf}
\usepackage[shortlabels]{enumitem}
\usepackage{tikz-cd} 
\usepackage{nopageno}

\usepackage{mathtools}
\usepackage{todonotes}

\newsiamremark{remark}{Remark}
\newsiamremark{hypothesis}{Hypothesis}
\crefname{hypothesis}{Hypothesis}{Hypotheses}
\newsiamremark{examples}{Examples}
\newsiamremark{question}{Question}
\crefname{question}{Question}{Question}

\newcommand{\cato}{\operatorname{C}}
\newcommand{\Coker}{\operatorname{Coker}}
\newcommand{\cmod}{\Psi}
\newcommand{\depth}{\operatorname{depth}}

\newcommand{\edim}{\operatorname{edim}}
\newcommand{\End}{\operatorname{End}}

\newcommand{\ext}{\operatorname{Ext}}
\newcommand{\fdim}{\operatorname{flat\,dim}}
\newcommand{\fitt}{\operatorname{Fitt}}
\newcommand{\fm}{{\mathfrak m}}

\newcommand{\fp}{{\mathfrak p}}
\newcommand{\hh}{\operatorname{H}}
\newcommand{\Hom}{\operatorname{Hom}}

\newcommand{\Ker}{\operatorname{Ker}}
\newcommand{\length}{\operatorname{length}}

\newcommand{\mbb}[1]{\mathbb{#1}}
\newcommand{\mco}{\mathcal O}
\newcommand{\pos}[1]{[\![{#1}]\!]}
\newcommand{\pdim}{\operatorname{proj\,dim}}
\newcommand{\rank}{\operatorname{rank}}
\newcommand{\Spec}{\operatorname{Spec}}
\newcommand{\tors}{\operatorname{tors}}
\newcommand{\tfree}[1]{{#1}^{\operatorname{tf}}}

\begin{document}

\title{\Large Commutative algebra inspired by modularity lifting}
    
    \author{Srikanth B. Iyengar\thanks{Department of Mathematics, University of Utah, Salt Lake City, UT, (\email{srikanth.b.iyengar@utah.edu}).}}

\date{\today}

\maketitle







\begin{abstract} 
This article gives an overview of some recent results in commutative algebra that are inspired by the work of Wiles, Taylor and Wiles, Diamond, Lenstra and others on the modularity of elliptic curves. 
\end{abstract}

\subsubsection*{Key words.}  congruence module, complete intersection ring, derived action, freeness criterion, Koszul complex, modularity lifting, Wiles defect.

\subsubsection*{AMS subject classifications.} 13C10, 13D02 (Primary), 11F80 (Secondary)

\section{Introduction.}
\label{sec:intro}
This paper may well have been titled: Commutative algebra inspired by Wiles, because all the mathematical developments described here have their origins in the proof due to Wiles~\cite{Wiles:1995}, and Taylor and Wiles~\cite{Taylor/Wiles:1995}, of the Shimura, Taniyama, Weil conjecture on the modularity of semistable elliptic curves over the rational numbers. This article is on the pure commutative algebraic results and research directions that have emerged from it, with a focus on my own joint work with  Brochard, Khare, and Manning over the past few years. I only hint at the number theoretic origins and motivations. In the companion article~\cite{Iyengar/Khare/Manning:2025a}, Khare, Manning, and I give an overview also of some of the results in number theory that we have been able to prove, or expect to prove, based on the new commutative algebra. As such, there is little overlap between the two articles, and I hope that, together, they convey a sense of this beautiful, and still evolving, mathematical landscape.

There are many excellent sources that give an overview of the work of Wiles' proof; see, for example,  \cite[Chapter 1]{Cornell/Silverman/Stevens:1997} and \cite{Darmon/Diamond/Taylor:1997}. Roughly speaking, using Mazur’s~\cite{Mazur:1997} deformation theory of Galois representations, Wiles reduces the Shimura, Taniyama, Weil conjecture to proving that certain surjective maps $\varphi_\Sigma\colon R_\Sigma \to \mbb T_\Sigma$ of (noetherian, commutative) local rings are isomorphisms. The ring $R_\Sigma$ parametrizes Galois deformations subject to constraints determined by a finite collection of primes $\Sigma$ and $\mbb{T}_\Sigma$ is a Hecke algebra subject to the same constraints, which acts on a certain space of modular forms, $M_\Sigma$. To prove that $\varphi_\Sigma$ is an isomorphism, Wiles first proves that $M_\Sigma$ is free as a $\mbb{T}_\Sigma$-module---this is a delicate argument—and also that $\mbb{T}_\Sigma$ is a complete intersection. He then uses these properties to prove that $\varphi_\Sigma$ is an isomorphism; as a corollary, he deduces the freeness of $M_\Sigma$ also as an $R_\Sigma$-module. It is Diamond’s insight~\cite{Diamond:1997} that one can prove the latter statement directly and that this implies $\varphi_\Sigma$ is an isomorphism of complete intersections, and also that $M_\Sigma$ is free over $\mbb{T}_\Sigma$.

The argument that $M_\Sigma$ is free over $R_\Sigma$ has two distinct parts. In the \emph{minimal case}, when $\Sigma=\varnothing$, the technique of ``patching"
introduced by Taylor and Wiles~\cite{Taylor/Wiles:1995} and adapted by Diamond~\cite{Diamond:1997}, yields a noetherian local ring $R_{\infty}$, a finitely generated $R_{\infty}$-module $M_{\infty}$, and surjective maps 
\[
R_{\infty} \twoheadrightarrow R_{\varnothing} \quad \text{and}\quad M_{\infty} \twoheadrightarrow M_{\varnothing}
\]
where the map on the left is a homomorphism of rings, say with kernel $I$, and $M_{\varnothing}\cong M_{\infty}/I M_{\infty}$. The point of the patching process is that $R_{\infty}$ is a regular local ring and $M_{\infty}$ is a maximal Cohen-Macaulay module over it. Thus the Auslander-Buchsbaum formula implies that $M_{\infty}$ is a free $R_{\infty}$-module, and hence that $M_{\varnothing}$ is a free $R_{\varnothing}$-module, as desired; see Proposition~\ref{pr:diamond}. It turns out that $I$ is generated by a regular sequence, so as a byproduct of the patching argument one gets that $R_{\varnothing}$ is a complete intersection; this is special to Wiles' context and important for the passage from the minimal to the general, non-minimal, case. This passage is facilitated by a new criterion, discovered by Diamond~\cite{Diamond:1997} (building on work of Wiles and Lenstra~\cite{Lenstra:1995}) that detects when a finitely generated module over a local ring is free \textbf{and} the ring is a complete intersection; see Theorem~\ref{th:diamond}.

Commutative algebra thus enters in both steps of the proof. As such, the result that Wiles and Diamond wanted (the Auslander-Buchsbaum formula) had been available since the late 1960's.  In contrast, to tackle the non-minimal case, Wiles and Diamond  needed to develop new commutative algebra, and in particular  a numerical criterion  (in terms that are meaningful in number theory) for detecting freeness of modules over local rings. Nevertheless, each step raises new questions, and leads to new results, in commutative algebra. This is what is described below.

The patching method is (or so it seems to me) specific to the number theoretic context. In Section~\ref{se:desmit}, I discuss the commutative algebra that is partly motivated by a desire to simplify, and even entirely do away with, the patching argument of Taylor and Wiles to establish modularity in the minimal case. It starts with a conjecture  of de Smit (now a theorem of Brochard's) that, in hindsight, can be interpreted as a converse to a well-known (to commutative algebraists!) result of Nagata on a special kind of surjective complete intersection map\footnote{However, the hypothesis of the de Smit conjecture provides the seed of a patching datum that develops into full fledged patching. The patching argument has become an indispensable tool in the number theorist's toolkit: it introduces a useful rigidity and robustness, allowing one to, for example, invert the prime $p$ when proving modularity of $p$-adic Galois representations, which the hypothesis of de Smit's conjecture does not provide.}.

The patching method of Taylor and Wiles has been, and continues to be, developed extensively, and is now available also in the non-minimal case, thanks to work of Kisin~\cite{Kisin:2009a}, and this has often allowed one to bypass the numerical criterion. However, although this has led to much more general modularity lifting theorems, among other applications, it has been at the expense of some refined torsion information; all this is explained in the introduction to \cite{Iyengar/Khare/Manning:2024a}. See also the discussion after Theorem~\ref{th:tate-wiebe}. The main contribution of \emph{op.~cit.} is to establish a numerical criterion that applies in greater generality, so that it can be applied to rings and modules \emph{after patching}. 

In a way, the key contribution of \cite{Iyengar/Khare/Manning:2024a} is the definition of a \emph{congruence module} attached to a module over a ring. This is recalled in Section~\ref{se:wilesdefect}, which is centered around  the numerical criterion of Diamond, Lenstra, and Wiles, and its generalization developed in \cite{Iyengar/Khare/Manning:2024a, Iyengar/Khare/Manning/Urban:2024}. In \cite{Iyengar/Khare/Manning:2025a},  the congruence module itself, and a closely related notion of the congruence ideal, take center stage---for they have taken an interesting life of their own---and we trace their origins in number theory and (potential) new applications.

\section{The de Smit conjecture and beyond.}
\label{se:desmit}
In what follows, I  assume familiarity with basic concepts from commutative algebra; much of what is needed can be found in Chapter 1 of the book by Bruns and Herzog~\cite{Bruns/Herzog:1998a}, and I take that  as the standard reference. The homological algebra notions that come up are already present, for most parts, in the book by Cartan and Eilenberg~\cite{Cartan/Eilenberg:1956a}. Given a local ring $A$, I write $\fm_A$ for its maximal ideal and $k_A$ for its residue field $A/\fm_A$. The \emph{embedding dimension of A}, denoted $\edim A$, is the $k_A$-rank of the Zariski tangent space $\fm_A/\fm_A^2$; equivalently, the minimal number of generators of the ideal $\fm_A$.

 The starting point of this section is the result below, abstracted from~\cite[Section~2]{Diamond:1997}. 

\begin{proposition}
\label{pr:diamond}
Let $\varphi\colon A\to B$ be a local homomorphism of noetherian local rings, with $A$ regular. Let $N$ be a nonzero $B$-module that is finitely generated as an $A$-module and satisfies $\pdim_AN\le \edim A-\edim B$, then $N$ is free as $B$-module and the local ring $B$ is regular as well.
\end{proposition}

\begin{proof}
By the Auslander-Buchsbaum formula~\cite[Theorem~1.3.3]{Bruns/Herzog:1998a} the hypotheses on $N$ implies the first inequality below
\[
\edim B \le \depth_A N \le \depth_B N \le \dim B\le\edim B\,.
\]
The other inequalities are standard; see ~\cite[\S1.2]{Bruns/Herzog:1998a}. Thus we conclude that $\edim B=\dim B$ and hence that $B$ is regular; see \cite[\S2.2]{Bruns/Herzog:1998a}. Moreover we get also that $\depth_BN=\dim B$ which, given that $B$ is regular, implies that $N$ is free as a $B$-module, again by the Auslander-Buchsbaum formula.
\end{proof}

Subsequent to Diamond's work, Bart de Smit conjectured that   if $A\to B$ is a local homomorphism of artinian local rings with $\edim A=\edim B$, then any $B$-module that is flat as an $A$-module is also  flat as a $B$-module. This strengthens Proposition~\ref{pr:diamond} in the case where $A$ and $B$ have the same embedding dimension, and allows one to dispense with the patching technique in the work of Wiles; see \cite[\S3]{Brochard:2017}.  In \cite[Theorem 1.1]{Brochard:2017}, Brochard verifies de Smit's conjecture by proving the stronger statement that if  $A,B$ are noetherian local rings with $\edim A\ge \edim B$  and $N$ is a  finitely generated $B$-module that is  flat over $A$,  then $N$ is flat over $B$. In   \cite{Brochard/Iyengar/Khare:2023a}, Brochard, Khare, and I generalize, with a simpler proof, Brochard's result as follows.

\begin{theorem}
\label{th:de-smit-general}
Suppose  $\varphi\colon A\to B$ is a local homomorphism of noetherian local rings, and $N$ is a nonzero finitely generated $B$-module whose flat dimension over $A$ satisfies $\fdim_AN\le \edim A-\edim B$, then $N$ is free as a $B$-module and the map $\varphi$ is  complete intersection with $\edim A - \dim A= \edim B -  \dim B$. \qed
\end{theorem}

Any local homomorphism $\varphi\colon A\to B$ can be factored, at least after completing $B$ at its maximal ideal, as a flat map with regular closed fiber, and a surjective map;  $\varphi$ is said to be \emph{complete intersection} if the kernel of the surjective map can be generated by a regular sequence. In particular,  when $\varphi\colon A\to B$ is surjective, it is a complete intersection if, and only if,  $\Ker\varphi$ can be generated by a regular sequence; see \cite[\S5]{Avramov:1999a} for details. 

 When $\varphi$ is a surjective complete intersection, the condition $\edim A-\dim A=\edim B-\dim B$, appearing in the statement above, is equivalent to the condition that $\Ker \varphi$ can be generated by a regular sequence whose image in  $\fm_A/\fm_A^2$ is a linearly independent set. There is a similar characterization of this property for general maps, in terms of Andr\'e-Quillen homology; see \cite{Brochard/Iyengar/Khare:2023a} for details. 

The  more general  hypothesis  in our theorem,  as opposed to \cite[Theorem 1.1]{Brochard:2017}, gives us the freedom to reduce its proof to the surjective case, using standard results in commutative algebra. This  is instrumental in leading to a proof that is simpler than  the one presented in \cite{Brochard:2017} as we can  apply an induction on $\edim A - \edim B$. A key ingredient in this induction step is the following theorem of Nagata.

\begin{theorem}
\label{th:nagata}
Let $A\to B$ be a surjective complete intersection with $\edim A-\dim A=\edim B-\dim B$. For any finitely generated $B$-module $N$ there is an equality
\[
\pdim_A N = \pdim_B N + \edim A - \edim B\,.
\]
\end{theorem}
See \cite[Corollary (27.5)]{Nagata:1962a}, or \cite[Proposition~3.3.5(1)]{Avramov:1998a}, for a proof. The main point is to prove that $\pdim_AN$ is finite if and only if $\pdim_BN$ is finite; the stated equality then follows from the Auslander-Buchsbaum formula.  Thus  Theorem~\ref{th:de-smit-general} can be seen as a converse to Nagata's theorem; keep in mind that flat dimension and projective dimension coincide for finitely generated modules over local rings.  See \cite{Iyengar/Letz/Liu/Pollitz:2022} for a perspective on Nagata's theorem and its converse through the lens of the structure of the derived category as a triangulated category.

\subsection*{Derived actions.}
Calegari and Geraghty \cite{Calegari/Geraghty:2018} extended the patching method  in  \cite{Taylor/Wiles:1995} to situations where one patches complexes rather than modules.  Based on the recasting of this method given in his work with Thorne~\cite{Khare/Thorne:2017}, Khare has proposed an extension of de Smit's conjecture that applies in this context. His question is best articulated in the language of \emph{derived actions} on complexes introduced in~\cite{Khare/Thorne:2017}, and recalled below.

 Let $A$ be a noetherian local ring and $B$ a (commutative) $A$-algebra. By a \emph{derived} $B$-complex $F$ (over $A$) we mean a finite free $A$-complex
\[
F\colonequals 0\longrightarrow F_n\longrightarrow \cdots \longrightarrow F_0\longrightarrow 0
\]
with $\hh_0(F)\ne 0$, equipped with a map of $A$-algebras $B\to \Hom_{\sf K(A)}(F,F)$, the homotopy classes of self-maps of the $A$-complex $F$; here $\sf K(A)$ is the homotopy category of complexes of $A$-modules. One could allow $F$ to be an arbitrary $A$-complex, but the  generality above suffices for the intended applications. Observe that then each  $\hh_i(F)$ is endowed with a $B$-module structure, extending that of $A$. However $F$ itself may not be equivalent to complex over $B$, even up to quasi-isomorphism; see the example below.  Any complex $F$ of $B$-modules is a derived $B$-complex, via the canonical map $B\to \Hom_A(F,F)$ given by multiplication. Here is a more interesting example.

\begin{examples}
   \label{ex:derived-complexes} 
Let $\mathbf{a}\colonequals a_1,\dots,a_n$ be elements in $A$ and $K$ the Koszul complex on $\mathbf{a}$. This is a finite free $A$-complex and for each $a_i$ the multiplication map $a_i\colon K\to K$ is homotopic to $0$; see \cite[Proposition~1.6.5]{Bruns/Herzog:1998a}. Thus the canonical map $A\to \Hom_{\sf K(A)}(K,K)$ factors through the quotient ring $B=A/(\mathbf{a})$ giving $K$ the structure of a derived $B$-complex. However, it is not always possible to realize $K$ as complex of $B$-modules; in that there may not be a complex of $B$-modules that is homotopic (or even quasi-isomorphic) to $K$.

For instance, when the sequence $\mathbf a$ generates the maximal ideal $\fm_A$ of the local ring $A$, the Koszul complex on $\mathbf{a}$ is quasi-isomorphic to a complex over the residue field $A/\fm_A$ of $A$, if and only if $A$ is regular; this is just a reformulation of the statement, due to Auslander and Buchsbaum, and Serre, that a local ring is regular if and only if its residue field has finite projective dimension. See also Theorem~\ref{th:rigid-action}
 below. \end{examples}

As I said earlier, complexes with derived actions come up naturally when one patches complexes rather than modules. Motivated in part by this, Khare posed the following question:

\begin{question}
\label{qu:Shekhar}
Let $A$ be a noetherian local ring and $F\colonequals 0\to F_n\to \cdots \to F_0\to 0$ a finite free $A$-complex with  $\hh_0(F)\ne 0$. If $F$ admits a structure of a derived $B$-complex where $B$ is  noetherian local $A$-algebra such that
\[
n \le  \edim A - \edim B \,,
\]
is then the $B$-module $\hh_{0}(F)$ free?
\end{question}

A positive answer to this question would obviate the need to patch in \cite{Calegari/Geraghty:2018}, by asking for patching data only for $n=1$,  just as de Smit's conjecture, proved in~\cite{Brochard:2017}, obviated the need to patch in the original argument of Taylor and Wiles; see \cite[Section~5]{Brochard/Iyengar/Khare:2025} for details.  Already in \cite{Calegari/Geraghty:2018} one finds a positive answer to Khare's question when the ring $A$ is regular; compare Proposition~\ref{pr:diamond}.  In ~\cite{Brochard/Iyengar/Khare:2025} Brochard, Khare, and I verify that Khare's question has a positive answer under additional hypotheses.

\begin{theorem}
\label{th:khare}
Question~\ref{qu:Shekhar} has a positive answer when  $B/\fm_AB$, the closed fiber of the map $A\to B$, is a complete intersection ring; this applies, in particular, when $A\to B$ is surjective.
\end{theorem}

In \cite[Theorem~6.1]{Brochard/Iyengar/Khare:2025} we prove more than what I stated above: When in addition the differential on $F$ satisfies $d(F)\subseteq \fm_AF$ (and this can always be arranged) there are also equalities
\[
\rank_AF_i =b\cdot \binom ni \quad\text{for $0\le i\le n$}
\]
where $b$ is the minimal number of generators of the $B$-module $\hh_0(F)$. The binomial coefficients on the right are precisely the ranks of the free modules appearing in a Koszul complex on $n$-elements. This is reminiscent of phenomena  well-known to commutative algebraists: Invariants attached to Koszul complexes often give good (typically, lower) bounds for general finite free complexes. Consider, for instance, the New Intersection Theorem~\cite{Roberts:1991}, the rank conjectures of Buchsbaum and Eisenbud~\cite{Buchsbaum/Eisenbud:1977}, and the Total Rank Conjecture, now a theorem proved by Walker~\cite{Walker:2017}, and VandeBogert and Walker~\cite{VandeBogert/Walker:2025}.
 
 It seems plausible that for a positive answer to Question~\ref{qu:Shekhar} one needs to input more information from number theory, by way of more structure on $F$. A natural strengthening of the hypothesis that $F$ admits a derived action of $B$ is that the action can be upgraded to a strict $B$-action, but this turns out to be too stringent. 

 \begin{theorem}
     \label{th:rigid-action}
     In the context of Question~\ref{qu:Shekhar}, if $F$ is quasi-isomorphic, in the derived category of $A$-modules,  to a complex of $B$-modules, then $\hh_i(F)=0$ for $i\ne 0$, the $B$-module $\hh_0(F)$ is free, and the map $A\to B$ is a  complete intersection with $\edim A-\edim B=\dim A - \dim B$.
 \end{theorem}

This is \cite[Theorem~1.2]{Brochard/Iyengar/Khare:2025}.  It applies, for example, when $F$ is the free resolution over $A$ of a finitely generated $B$-module, and hence it subsumes Theorem~\ref{th:de-smit-general}. Nevertheless, the conclusion that the homology of $F$ is concentrated in degree $0$ means that the hypothesis cannot always apply in the number theoretic context where one patches complexes, for the complexes involved  typically compute homology of arithmetic manifolds, and can have homology in positive degrees as well; see \cite[Section~5]{Brochard/Iyengar/Khare:2025}.

Khare's question also suggests a systematic study of invariants attached to complexes with derived actions of commutative rings. For instance: What restrictions arise on, say, the Betti numbers of a finite free $A$-complex $F$ if it admits a derived $B$-action? See \cite[Lemma~10.7]{Iyengar/Khare/Manning:2024a} for one such restriction.

\section{Wiles defect.}
\label{se:wilesdefect}
This section focuses on the commutative algebra that has emerged from the numerical criterion for detecting when a surjective map of local rings is an isomorphism of complete intersections, due to Wiles~\cite{Wiles:1995}, Diamond~\cite{Diamond:1997}, and Lenstra~\cite{Lenstra:1995}. It is a key tool in the argument that modularity propagates from the minimal to the non-minimal case;  see \cite[Introduction]{Iyengar/Khare/Manning:2024a} for details and a historical context for these ideas. 

Throughout this section $\mco$ is a complete discrete valuation ring and $A$ is a complete local $\mco$-algebra; thus, of the  form $\mco\pos{x_1,\dots,x_n}/I$. It is helpful to call such a ring a \emph{local analytic $\mco$-algebra}; see Kunz~\cite[\S13]{Kunz:1986} or \cite{Berger/Kiehl/Kunz/Nastold:1967}. It also helps to introduce the following language: An \emph{$\mco$-valued point} in $\Spec A$ is a map of $\mco$-algebras $\lambda\colon A\to \mco$; the \emph{codimension} of this point is the height of the prime ideal $\fp\colonequals \Ker\lambda$. In what follows, I say that  $A$, or an $A$-module $M$, has a property \emph{at $\lambda$} if the ring $A_\fp$, or the $A_\fp$-module $M_\fp$, has that property. For instance, when I write ``$M$ is free (of rank $\mu$) at $\lambda$", I mean that the $A_\fp$-module $M_\fp$ is free (of rank $\mu$).

Wiles et.~al.\ consider $\mco$-valued points $\lambda\colon A\to \mco$ of codimension $0$ such that, for $\fp=\Ker \lambda$, the map 
\[
\lambda_\fp \colon A_\fp \longrightarrow \mco_\fp\,,
\]
induced by localization at $\fp$ is an isomorphism. Note that $\mco_\fp$ is the field of fractions of $\mco$; in particular, $A$ is regular at $\lambda$. The idea is to probe properties of $A$, or some $A$-module $M$, by testing its properties at suitable regular points of codimension $0$. 

As noted above, in the context of modularity lifting, the goal is to prove that some $A$-module $M$ is free, or at least has a nonzero free summand, which suffices for many applications. To that end, one considers the conormal  module of the map $\lambda$, namely the $\mco$-module $\fp/\fp^2$, and, for any finitely generated $A$-module $M$, the \emph{congruence module} of $M$, defined to be
\begin{equation}
\label{eq:cmod0}
\cmod_\lambda(M)\colonequals \frac{M}{M[\fp]+M[I]} \quad \text{where $I=A[\fp]$.}    
\end{equation}
Here $M[I]$ denotes the submodule $\{m\in M\mid I\cdot m=0\}$ consisting of elements of $M$ annihilated by $I$. Number theoretic origins and motivations for this definition are given in the Introduction to \cite{Iyengar/Khare/Manning:2025a}. The hypothesis that $\lambda$ is a regular point of codimension 0 ensures that both $\mco$-modules $\fp/\fp^2$ and $\cmod_\lambda(M)$ have finite length; this is easy to check, by localizing at $\fp$. It is equally easy to check that also the converse holds: If for some $\mco$-valued point $\lambda$, either $\fp/\fp^2$ or $\cmod_\lambda(A)$ has finite length, then $\lambda$ is regular of codimension $0$. See \cite[\S5.1]{Darmon/Diamond/Taylor:1997} or \cite[Example~2.7]{Iyengar/Khare/Manning:2025a} for examples arising from number theory.

The following result is essentially due to Diamond~\cite{Diamond:1997} and Wiles~\cite{Wiles:1995}. I have incorporated improvements due to Fakhruddin, Khare, and Ramakrishna~\cite[Appendix]{Fakhruddin/Khare/Ramakrishna:2021}, and stated it using the language introduced above.
By ``$M$ is supported at $\lambda$" I mean that $M_\fp\ne 0$ for $\fp\colonequals \Ker\lambda$.

 \begin{theorem}
     \label{th:diamond}
Let $A$ be a local analytic $\mco$-algebra and $M$ a finitely generated $A$-module with $\depth_AM\ge 1$.
If $M$ is supported at a regular $\mco$-valued point  $\lambda$ of codimension $0$,  then $M$ is free at $\lambda$, say of rank $\mu$, and for $\fp\colonequals \Ker\lambda$ there is an inequality
\[
\mu\cdot \length_\mco (\fp/\fp^2) \ge \length_\mco \cmod_\lambda(M)\,.
\]
Equality holds if and only if $A$ is complete intersection and $M\cong A^\mu\oplus W$, with $W$  not supported at $\lambda$. \qed
 \end{theorem}

Let me make a few comments on the special case when $M=A$; in any case, this is an important step in the proof of the general case of theorem; see the discussion around Theorem~\ref{th:splitting}.

In the sequel, the $i$'th Fitting ideal of a finitely generated module $N$ over a ring $R$ is denoted $\fitt^R_i(N)$; see \cite[pp.~21]{Bruns/Herzog:1998a} or \cite[Appendix D]{Kunz:1986} for the definition and basic properties of Fitting ideals. The ring is usually suppressed from the notation. Given that $\mco$ is a discrete valuation ring,  the length of a finitely generated $\mco$-module, say $U$, is equal to the length of $\mco/\fitt_0(U)$.

The $0$'th Fitting ideal of a module annihilates it, so one gets an inclusion
\begin{equation}
\label{eq:wiles-inequality}
\fitt_0(\fp) \subseteq A[\fp]\,,
\end{equation}
and this leads to the inequality in Theorem~\ref{th:diamond}, for $M=A$. Here is a slightly different perspective on the inclusion in \eqref{eq:wiles-inequality}: Let $\mathbf{a}\colonequals a_1,\dots,a_n$ be a minimal generating set for the ideal $\fp=\Ker\lambda$, and $K$ the Koszul complex on $\mathbf{a}$, viewed as a dg algebra; see \cite[Proposition~1.6.2]{Bruns/Herzog:1998a}. Its homology $\hh_*(K)$ is a strictly-graded commutative algebra. This gives the inclusion below:
\[
 \fitt_0(\fp) K_n  = \wedge^n \hh_1(K) \subseteq \hh_n(K) = A[\fp]K_n\,. 
\]
The equalities are immediate from an inspection of the differentials in the Koszul complex and the definition of Fitting ideals. Since $K_n\cong A$, the inclusion above is equivalent to the one in \eqref{eq:wiles-inequality}. Therefore Theorem~\ref{th:diamond}, for $M=A$, translates to the following:

\begin{theorem}
\label{th:dlw}
Let $A$ be a local analytic $\mco$-algebra with $\depth A\ge 1$ and $\lambda$ a regular $\mco$-valued point of codimension $0$. Let $K$ be the Koszul complex on a minimal generating set for $\Ker\lambda$, say consisting of $n$ elements. Then $\wedge^n\hh_1(K)=\hh_n(K)$ if and only if $A$ is complete intersection and $\dim A=1$.
\end{theorem}

The claim that the stated equality holds when $A$ is a complete intersection follows from a result of Tate~\cite[Theorem~4]{Tate:1957a}.
A somewhat different argument, tailored to the specific context, is given in \cite[\S5.3]{Darmon/Diamond/Taylor:1997}; see also the computation in the Appendix to ~\cite{Mazur/Roberts:1969}, by Mazur and Roberts; both are attributed to Tate.

The reverse implication, that when $\wedge^n\hh_1(K)=\hh_n(K)$ holds $A$ is complete intersection,  is reminiscent of, and easily deduced from, the result below due to Wiebe~\cite[Theorem 2.3.16]{Bruns/Herzog:1998a}.

\begin{theorem}
\label{th:tate-wiebe}
Let $R$ be a noetherian local ring and $K$ the Koszul complex on a minimal generating set for the maximal ideal of $R$. With $e=\edim R$, if $\wedge^e \hh_1(K)\ne 0$, then $R$ is an artinian complete intersection. \qed
\end{theorem}

Number theoretic applications of the numerical criterion of Diamond, Lenstra, and Wiles rely on ingredients (like Ihara's Lemma) that proved hard to generalize. More fundamental than this pragmatic difficulty is the point that the numerical criterion proves isomorphisms of complete intersections, and in later developments of Wiles' methods the rings are not expected to be complete intersections. There are also situations where the rings are torsion and hence admit no $\mco$-valued points. Applying a numerical criterion to patched rings circumvents these problems as these are expected to be complete intersections, with sufficiently many regular $\mco$-valued points, in more complex cases as well. However, the points are no longer of codimension zero, so one need extensions of the theory of congruence modules and Wiles defect that apply to points of higher codimension. This is achieved in \cite{Iyengar/Khare/Manning:2024a}. What follows is an overview of our findings in \emph{op.~cit.}, and also of some ongoing projects that  build on it.

\subsection*{Higher codimensions.}
In what follows $\cato_{\mco}$ denotes the category of such pairs $(A,\lambda)$ where $A$ is a local analytic $\mco$-algebra and $\lambda\colon A\to \mco$ is a regular $\mco$-valued point in $\Spec A$; there is no restrictions on the codimension of the point. The morphisms in $\cato_\mco$ are the morphisms of $\mco$-algebras over $\mco$, that is to say morphisms $\varphi \colon A\to B$ of $\mco$-algebras such that the following diagram commutes:
\[
\begin{tikzcd}
A \arrow[dr,"\lambda_A" swap] \arrow[rr,"\varphi"] && B \arrow[dl,"\lambda_B"] \\
&\mco&
\end{tikzcd}
\]
The subcategory consisting of pairs $(A,\lambda)$ where $\lambda$ (that is to say, $\Ker\lambda$) has codimension $c$ is denoted $\cato_\mco(c)$. The work of Wiles et.~al. discussed above is all about the category $\cato_\mco(0)$.

Fix $(A,\lambda)$ in $\cato_\mco(c)$, and set $\fp\colonequals \Ker\lambda$. The hypothesis that $\lambda$ is regular of codimension $c$ translates to the condition that the conormal module $\fp/\fp^2$ has rank $c$, that is to say, isomorphic to $\mco^c\oplus \tors(\fp/\fp^2)$. I write $\tors U$ for the torsion submodule of an $\mco$-module $U$; thus, $\tors U=\Ker (U\to U\otimes_\mco E)$, where $E$ is the field of fractions of $\mco$. The cokernel of the inclusion $\tors U\subseteq U$ is the \emph{torsion-free} quotient of $U$, so one has an exact sequence 
\[
0\longrightarrow \tors U \longrightarrow U \longrightarrow \tfree U\longrightarrow 0
\]
of $\mco$-modules. This is split-exact, which is important in some arguments, though there is no canonical splitting. 

The torsion submodule of the conormal module $\fp/\fp^2$ (rather, its length) plays an important role in the sequel. Observe that the length of $\tors(\fp/\fp^2)$ equals the length of $\mco/\fitt_c(\fp/\fp^2)$. It follows from the Jacobi-Zariski sequence arising from the diagram $\mco\to A\xrightarrow{\lambda}\mco$ that there is a natural isomorphism
\begin{equation}
\label{eq:kahler-different}
\hat{\Omega}_{A/\mco}\otimes_A \mco \cong \fp/\fp^2
\end{equation}
where $\hat{\Omega}_{A/\mco}$ is the universally finite module of  differentials; see \cite[\S11]{Kunz:1986}.  Hence $\fitt_c(\fp/\fp^2)$ equals 
$\fitt^A_c(\hat{\Omega}_{A/\mco})\cdot \mco$, the $c$'th Fitting ideal of the module of K\"ahler differentials extended to $\mco$ along $\lambda$. The relevance of this observation is made clear in \cite[\S2.9]{Iyengar/Khare/Manning:2025a}.

Let $M$ be a finitely generated $A$-module. The surjective map $M\to M/\fp M$ induces the first map below:
\[
\eta_\lambda(M)\colon \ext^c_A(\mco,M)\longrightarrow \ext^c_A(\mco,M/\fp M)\longrightarrow \tfree{\ext^c_A(\mco,M/\fp M)}\,.
\]
The second is the natural quotient map.  The \emph{congruence module} of $M$ at $\lambda$ is the cokernel of the map above:
\[
\cmod_\lambda(M) \colonequals \Coker \eta_\lambda(M)\,.
\]
 The appearance of the torsion-free quotient in the definition of $\eta_\lambda(M)$, and hence in the congruence module, seems a little arbitrary, but it is exactly what is needed to make this a good definition, as becomes clear from the proof of the results in \cite[Part~1]{Iyengar/Khare/Manning:2024a}. It is also needed to ensure that for $c=0$ this module is the same as the one given in \eqref{eq:cmod0} as long as $\depth_AM\ge 1$; verifying this is a simple exercise, done in \cite[Proposition~2.10]{Iyengar/Khare/Manning:2024a}.
 
 Once again, since $\lambda$ is regular of codimension $c$, the congruence module of $M$ has finite length for each $M$. Conversely, the length of $\cmod_\lambda(A)$ can be finite only when $\lambda$ is  a regular  point. Unlike in the case $c=0$, the proof of this claim depends on a difficult result, due to Lescot~\cite{Lescot:1983}; see \cite[Theorem~2.1 and Notes 2.1.1]{Iyengar/Khare/Manning:2025a}.
 
 There is a closely related notion of a \emph{congruence ideal} attached to $M$; please see \cite{Iyengar/Khare/Manning:2025a} for details, including a description of the map $\eta_\lambda(M)$ in term of Yoneda extensions, and for the historical antecedents for these notions.  In this paper, I will focus mainly on the following invariant 
\begin{equation}
\label{eq:defect}
\delta_{\lambda}(M)\colonequals 
    \mu\cdot \length_{\mco} \tors(\fp/\fp^2)  - \length_{\mco}\cmod_\lambda(M)\,,
\end{equation}
where $\mu$ is the rank of $M$ at $\lambda$ and  $\fp=\Ker\lambda$, introduced in \cite{Iyengar/Khare/Manning:2024a} as the \emph{Wiles defect of $M$} with respect to $\lambda$; compare \eqref{eq:wiles-inequality}, keeping in mind that $\fp/\fp^2$ is torsion when $c=0$. It follows from Theorem~\ref{th:diamond} that $\delta_{\lambda}(M)\ge 0$ for $c=0$. It is not clear, and I suspect it is not true in general, that the defect is always non-negative without additional hypotheses on $M$; see Theorem~\ref{th:splitting} below.

To understand the Wiles defect better it helps to consider the map
\begin{equation}
    \label{eq:kappaM}
    \kappa_\lambda(M)\colon \ext^c_A(\mco,A)\otimes_A M \longrightarrow \ext_A^c(\mco,M) \longrightarrow \tfree{\ext_A^c(\mco,M)}\,,
\end{equation}
where the one on the left is the usual K\"unneth map. The K\"unneth map appears in the long exact sequence relating the bounded Ext-modules, the Ext-modules, and the stable Ext-modules, of the pair $(\mco,M)$; see \cite[Notes 2.4.1]{Iyengar/Khare/Manning:2025a}. This partly explains its relevance to detecting free summands of $M$; see Theorem~\ref{th:splitting} below.

Most results about congruence modules require that $M$ has sufficient depth; usually $\depth_AM\ge c+1$ suffices. Such a hypothesis is reasonable in number theoretic applications, where one has often good control on the modules $M$ that appear, but not always on the underlying ring $A$.  The conditions $\depth_AM\ge c+1$ has two useful consequences:
  \begin{itemize}
      \item 
      The $\mco$-module $\ext^c_{A}(\mco,M)$ is torsion-free, so \eqref{eq:kappaM} is the usual K\"unneth map
      \item 
      The module $M$ is free at any regular $\mco$-valued point $\lambda$.
  \end{itemize}
 The first statement is immediate from the long exact sequence that arises when we apply $\Hom_A(-,M)$ to the exact sequence of $A$-modules $0\to \mco\xrightarrow{\varpi}\mco\to k_A\to 0$, where $\varpi$ is a uniformizer for $\mco$. The second statement follows from the Auslander-Buchsbaum formula, applied to the $A_\fp$-module $M_\fp$.

The following ``defect formula" from \cite[Lemma~3.7]{Iyengar/Khare/Manning:2024a} teases apart the defect of module into one part that depends only on the ring $A$, and another part coming from \eqref{eq:kappaM}, and hence dependent on $M$:
\begin{equation}
\label{eq:defect-formula}
\delta_{\lambda}(M) = \mu\cdot \delta_{\lambda}(A) + \length_{\mco}\Coker \kappa_\lambda(M)\,.     
\end{equation}
This holds under the assumption that $M$ is free, of rank $\mu$, at $\lambda$.
This goes into the proof of the following result from \cite{Iyengar/Khare/Manning:2025a}; it is a perfect extension of Theorem~\ref{th:diamond}.

\begin{theorem}
    \label{th:splitting}
Let $A$ be an analytic $\mco$-algebra and $M$ a finitely generated $A$-module.
If $M$ is supported at a regular $\mco$-valued point  $\lambda$ of codimension $c$ and $\depth_A M\ge c+1$, then
$\delta_\lambda(M)\ge 0$. Moreover, $\delta_\lambda(M)=0$ if and only if $A$ is complete intersection and $M\cong A^\mu\oplus W$, where $\mu$ is the rank of $M$ at $\lambda$, and $W$ is not supported at $\lambda$. \qed   
\end{theorem}

The argument invokes the following result concerning deformation in $\cato_\mco$.

\begin{theorem}
    \label{th:deformation}
Fix $(A,\lambda_A)$ in $\cato_\mco(c)$ and a finitely generated $A$-module $M$. For $\fp\colonequals \Ker\lambda$, let $f$ in $ \fp\setminus \fp^{(2)}$ be element that is not a zero divisor on $M$. Set $B\colonequals A/Af$ and let $\lambda_B\colon B\to \mco$ the induced map. Then $(B,\lambda_B)$ is in $\cato_{\mco}(c-1)$, and for the $B$-module $N\colonequals M/fM$ there is an equality
\[
\delta_{\lambda_B}(N) = \delta_{\lambda_A}(M) \,.
\]
\end{theorem}

Here $\fp^{(2)}$ is the second symbolic power of $\fp$, namely $\fp^2A_\fp\cap A$. This theorem is proved by verifying that both components in the definition of the Wiles defect~\eqref{eq:defect} change by the same amount, namely the order of the class of $f$, in the conormal module of $\lambda_A$; see \cite[\S8]{Iyengar/Khare/Manning:2024a} and, for slightly different perspective, \cite[\S2]{Iyengar/Khare/Manning/Urban:2024}.

Theorem~\ref{th:splitting} is proved by an induction on $c$; the base case $c=0$ is precisely Theorem~\ref{th:diamond}. When $c\ge 1$, since $\depth_AM\ge c+1$, one can find an element $f$ as in Theorem~\ref{th:deformation} and deduce by induction that the ring $B$ is complete intersection, and that $N$ has a free summand of rank $\mu$. However, to propagate this information to $A$ and $M$, one needs to know that $f$ is also not a zerodivisor on $A$, and so, eventually, that $A$ has sufficient depth. However, as I mentioned earlier, in the number theoretic context, this is not always a reasonable hypothesis. To get around this, we argue by first replacing $A$ by its image in $\End_A(M)$. This step is a little delicate and requires an ``invariance of domain" property of congruence modules; see the proof of \cite[Theorem~7.4]{Iyengar/Khare/Manning:2024a}.

Theorem~\ref{th:splitting} vastly expands the reach of the patching method, for it can be applied after patching, and this leads to new results from number theory. One such is \cite[Theorem~F]{Iyengar/Khare/Manning:2024a}, which is proved under some restrictive hypotheses. In ongoing joint work with Diamond~\cite{Diamond/Iyengar/Khare/Manning:2025a} we expect to get significantly better and more complete results; these are discussed in \cite[3.2]{Iyengar/Khare/Manning:2025a}.

\subsection*{Infinitely generated modules.}
 In many applications, the module $M$ that appears is not finitely generated over $A$; one can restore finite generation by taking into account the action of a profinite group that acts naturally on $M$. At the very least, this calls for establishing analogues of Theorems~\ref{th:splitting} and \ref{th:deformation} to allow (certain types of) infinitely generated modules. Here is a result from an ongoing collaboration~\cite{Allen/Iyengar/Khare/Manning:2025a} with Allen, Khare, and Manning that gives an idea of what is needed; we use it to obtain new results on completed cohomology modules, in the sense of Emerton and Gee~\cite{Calegari/Emerton:2012}.

\begin{theorem}
        \label{th:bigdeal}
Fix $(A,\lambda)$ in $\cato_{\mco}(c)$, and suppose also that $A$ is Gorenstein. Let $M$ be an $\fm_A$-adically complete $A$-module with $\depth_AM =\dim A$. If the map $\kappa_\lambda(M)$ from \eqref{eq:kappaM} is surjective, then $M\cong F\oplus W$ where $F$ is flat and $\fp$ is not in the support of $W$. \qed
    \end{theorem}

An $A$-module $M$ with the property that $\depth_AM=\dim A$ is said to be a \emph{big Cohen-Macaulay module}.
It is well-understood that  completion is often a good substitute for finite generation; Nakayama's Lemma holds, for instance. However, we have not been able to verify the statement above by mimicking the argument for finitely generated modules given in \cite[Theorem~9.2]{Iyengar/Khare/Manning:2024a}. Instead, the proof we found uses in an essential way the theory of derived completions initiated by Greenlees and May~\cite{Greenlees/May:1992a}, and is more satisfactory than the one in \cite{Iyengar/Khare/Manning:2024a}, for it gives an explicit splitting of $M$. 

Number theory demands more (she \emph{is} the Queen of Mathematics): One would like the Gorenstein property (in fact, that the ring is complete intersection) as a conclusion of the numerical criterion, and not as part of the hypotheses; compare Theorem~\ref{th:splitting}. This leads to the next topic.

\subsection*{Wiles defect module.}
In Theorem~\ref{th:splitting} the hypothesis is only that $\delta_\lambda(M)=0$, that is to say that
\begin{equation}
\label{eq:delta=0}
\mu\cdot \length_{\mco} \tors(\fp/\fp^2) = \length_{\mco} \Psi_{\lambda}(M)\,.
\end{equation}
The statement is agnostic to the values of the two invariants above. Thus, if one can construct a defect module the criteria for freeness (and complete intersection) described above can be phrased in terms of the vanishing of the defect module, and hence is sensible even when $M$ and the defect module are not  finitely generated. Building on the discussion around Theorems~\ref{th:dlw} and \ref{th:tate-wiebe},  we have identified a potential candidate for the defect module: For $(A,\lambda)$ in $\cato_{\mco}(c)$ let $K$ be the Koszul complex on some finite generating set $a_1,\dots,a_n$ for the ideal $\Ker \lambda$. With $c$ the codimension of $\lambda$, the defect module of a module $M$ (not necessarily finitely generated) should be the cokernel of the natural map
\[
\bigwedge^{n-c}_{\mco}\hh_1(K)\otimes_{\mco} \hh_0(K\otimes_A M) \longrightarrow \hh_{n-c}(K\otimes_AM)
\]
induced by the action of $\hh_*(K)$ on $\hh_*(K\otimes_AM)$. This definition is also inspired by various criteria for detecting the complete intersection property of a local ring in terms of the multiplicative structure of the Koszul homology of the ring with respect to its maximal ideal, due to Assmus, Bruns, Tate, and Wiebe; see~ Theorem~\ref{th:tate-wiebe} and also \cite[\S2.3]{Bruns/Herzog:1998a}. We can verify, at least when both $A$ and $M$ have sufficient depth, that this defect module has the desired properties, and also that, when $M$ is in addition finitely generated, the length of the defect module is precisely $\delta_\lambda(M)$, the Wiles defect of $M$. Such a Koszul homology description of the defect module also makes it easier to compute the Wiles defect itself, in case $M$ is finitely generated. However, much remains to be done.

\section*{Acknowledgments.}
This work is partly supported by National Science Foundation grant DMS-2502004.  Wiles' work appeared when I was in graduate school, at Purdue University, West Lafayette. I vividly recall the excitement it generated, and a seminar (given by Craig Huneke) exposing the commutative algebra aspects of Wiles' work. I could not have dreamed that, one day, my own work would become intertwined with this thread of research; I have to thank Shekhar Khare for this. A chance conversation on the UCLA campus, around November 2019, aided by the long shutdown the following year, sparked the various projects that I am reporting here and in \cite{Iyengar/Khare/Manning:2025a}. Many thanks to him, and our collaborator Jeff Manning, for generously sharing their expertise and  inviting me along on this wonderful journey. 

I have been fortunate in my collaborations, and I would  like to take this opportunity to thank in  particular Dave Benson, Henning Krause, and Julia Pevtsova, and also Mark Walker, for the mathematical conversations we have been having, almost weekly, for years. A special thanks to Lucho Avramov, who was my Ph.\ D.\ advisor, and John Greenlees, who was my post-doc supervisor. They will  recognize many of the techniques, insights, and points of view, that go into this work, for I learned these from them.

\bibliographystyle{siamplain}
\newcommand{\noopsort}[1]{}

\end{document}